\documentclass[11pt]{article}
\usepackage{epic,latexsym,amssymb}
\usepackage{color}
\usepackage{tikz}
\usepackage{amsmath,amsthm}
\usepackage{dirtytalk}
\usepackage{multirow}
\usepackage{pdfpages}
\usepackage{fancyhdr}

\newtheorem{thm}{Theorem}
\newtheorem{lem}{Lemma}

\newtheorem{rem}{Remark}

\newtheorem{prop}{Proposition}

\newtheorem{defn}{Definition}

\newcommand{\Zekhaya}[1]{{\color{red} #1}}

\let\oldenumerate\enumerate
\renewcommand{\enumerate}{
  \oldenumerate
  \setlength{\itemsep}{0pt}
  \setlength{\parskip}{0pt}
  \setlength{\parsep}{0pt}
}

\def\vertex(#1){\put(#1){\circle*{2}}}
\def\vertexo(#1){\put(#1){\circle{2}}}
\def\vert(#1){\put(#1){\circle*{1.5}}}
\def\verto(#1){\put(#1){\circle{1.5}}}
\def\lab(#1)#2{\put(#1){\makebox(0,0)[c]{#2}}}

\begin{document}

\title{Nordhaus-Gaddum inequalities for the number of $1$-nearly independent vertex subsets}

\author{Eric O. D. Andriantiana \\
	Department of Mathematics (Pure and Applied) \\
	Rhodes University \\
	Makhanda, 6140 South Africa\\
    and National Institute for Theoretical and Computational Sciences (NITheCS)\\
    Makhanda, South Africa\\
	\small \tt Email: e.andriantiana@ru.ac.za \\\\Zekhaya B. Shozi\\
	School of Science \& Agriculture\\
	University of KwaZulu-Natal\\
	Durban, 4000 South Africa\\
    and National Institute for Theoretical and Computational Sciences (NITheCS)\\
    Durban, South Africa\\
\small \tt Email: shoziz1@ukzn.ac.za\\
\small \tt Email: zekhaya@aims.ac.za
}

\date{}
\maketitle

\begin{abstract}
For a graph $G$, a vertex subset is called \emph{$1$-nearly independent} if the subgraph it induces contains exactly one edge. Let $\sigma_1(G)$ denote the number of such subsets in $G$. In this paper, we study Nordhaus-Gaddum type inequalities for $\sigma_1$, that is, bounds on the sum $\sigma_1(G)+\sigma_1(\overline{G})$, where $\overline{G}$ denotes the complement of $G$. 

We establish that, for any $n$-vertex graph $G$, we have
$\sigma_1(G)+\sigma_1(\overline{G})\geq n(n-1)/2,$
with equality if and only if $G$ is either complete or edgeless. We further obtain that among all trees of order $n$, the star $K_{1,n-1}$ uniquely minimises $\sigma_1(T)+\sigma_1(\overline{T})$. Finally, we prove that for all graphs of order $n \ge 6$,
\[
\sigma_1(G)+\sigma_1(\overline{G}) \le \frac{27}{64}\,2^{n} + \frac{1}{2}(n+2)(n-3),
\]
with equality if and only if $G$ or $\overline{G}$ is isomorphic to $3K_2 \cup \overline{K_{n-6}}$.
\end{abstract}

\newpage
	
\section{Introduction}

A \emph{simple} and \emph{undirected} graph $G$ is an ordered pair of sets $(V(G),E(G))$, where $V(G)$ is a nonempty set of elements which are called \emph{vertices} and $E(G)$ is a (possibly empty) set  of $2$-element subsets of $V(G)$ which are called \emph{edges}. The number of elements in $V(G)$ is the \emph{order} of $G$, while the number of elements in $E(G)$ is the \emph{size} of $G$. {$\overline{G}$ denotes the complement of $G$, which is  defined as $\overline{G} = (V(G), \{uv \mid u,v \in V(G), u\ne v \text{ and } uv\notin E(G)\})$. } For graph theory notation and terminology, we generally follow~\cite{henning2013total}. 

A new generalisation of independent vertex subsets has recently been proposed \cite{andriantiana2024number}. For an integer $k\ge 0$, a subset $I_k$ of the vertices of a graph $G$ is a \emph{$k$-nearly independent vertex subset} of $G$ if the subgraph induced by $I_k$ in $G$ contains exactly $k$ edges. Denote by $\sigma_k(G)$ the number of $k$-nearly independent vertex subsets of a graph $G$. Thus, $\sigma_0(G)$ is the number of subsets that are pairwise nonadjacent in $G$, also known as the \emph{Merrifield-Simmons index} of $G$ (see \cite{Merrifield198055}). 

Among various findings in \cite{andriantiana2024number}, it is shown that while the star $K_{1,n-1}$ has the largest $\sigma_0$ among all trees (and connected graphs) with $n$ vertices, it has the smallest $\sigma_1$. It is also revealed in the paper that the largest value of $\sigma_1$ among all graphs with $n\geq 6$ vertices is only reached by three or four copies of $K_2$ completed with isolated vertices.

A subset $S\subseteq E(G)$ is called a 1-nearly independent edge subset if exactly one pair of edges in $S$ share a common endpoint. The function $Z_1(G)$ counts such subsets. The study of $Z_1$ has been initiated in \cite{Andriantiana2025NumberOf1Nearly}, providing a characterisation of extremal trees and graphs.


Inequalities that relate the sum or product of a graph invariant to the same invariant of its complement are usually referred to as Nordhaus-Gaddum inequalities. It is a fairly popular topic in extremal graph theory, see the survey \cite{AouchicheHansen2013Survey,goddard1992some, harary1996nordhaus} and the recent papers \cite{FaughtKemptonKnudson2024,henning2011nordhaus, hernando2014nordhaus, LiGuo2024LaplacianKthNG}. In \cite{andriantiana2022nordhaus}, for example, Nordhaus-Gaddum type bounds for the number, $\eta$, of connected induced subgraphs in graphs is studied. They proved that  among all graphs $G$ with $n$ vertices, $\eta(G) + \eta(\overline{G})$ is minimum if and only if $G$ has no induced path on four vertices. Also, among all trees $T$ with $n$ vertices, $\eta(T) + \eta(\overline{T})$ is minimum if and only if $T$ is the star $K_{1,n-1}$. The tree that attains the maximum $\eta(T) + \eta(\overline{T})$ is the tree whose degree sequence is $(\lceil n/2 \rceil, \lfloor n/2 \rfloor, 1, \ldots, 1)$.

In this paper, we study Nordhaus-Gaddum inequalities for $\sigma_1$. We determine bounds for $\sigma_1(G) + \sigma_1(\overline{G})$. The rest of the paper is structured as follows. Section \ref{sec:preliminary} is the preliminary section, where we provide some graph-theoretic terminology and notation which we shall adhere to throughout the paper. Also, the explicit formulas for $\sigma_1(G) + \sigma_1(\overline{G})$ is presented for some selected classes of graphs $G$ in terms of their order $n$. Our main results are in Section \ref{sec:main-result}. Firstly, in Section \ref{subsec:minimal-graphs}, we present a tight lower bound on $\sigma_1(G) + \sigma_1(\overline{G})$ if $G$ is a general graph with $n$ vertices. We also show that this bound is attained if $G$ is an edgeless graph $\overline{K_n}$ or if $G$ is a complete graph $K_n$. Secondly, in Section \ref{subsec:Minimal-trees}, we present a tight lower bound on $\sigma_1(T) + \sigma_1(\overline{T})$ if $T$ is a tree with $n$ vertices. We show that this bound is uniquely attained by the star $K_{1,n-1}$. Lastly, in Section \ref{subs:maximal-general-graphs}, we present a tight upper bound on $\sigma_1(G) + \sigma_1(\overline{G})$ if $G$ is a general graph with $n$ vertices. We show that this bound is attained if $G$ or $\overline{G}$ is the union of three copies of $K_2$ and $n-6$ copies of $K_1$, i.e., $3K_2 \cup \overline{K_{n-6}}$.

\section{Preliminary}
\label{sec:preliminary}

 Let $G$ be a graph. 
 We denote the degree of a vertex $v$ in $G$ by $\deg_G(v)$ or simply by $\deg(v)$ if there is no risk of confusion. 
 The \emph{open neighbourhood} of a vertex $v$ in $G$ is $N_G(v) = \{ u \in V(G)  \mid uv \in E(G)\}$, while its \emph{closed neighbourhood} is $N_G[v] = \{v\} \cup N_G(v)$. Again, if there is no risk of confusion, we write $N(v)$ rather than $N_G(v)$ and $N[v]$ rather than $N_G[v]$. For a subset $S$ of vertices of a graph $G$, we denote by $G - S$ the graph obtained from $G$ by deleting the vertices in $S$ and all edges incident to them. If $S = \{v\}$, then we simply write $G - v$ rather than $G - \{v\}$. 
 We denote the complete graph on $n$ vertices by $K_n$. For positive integers $r$ and $s$, we denote by $K_{r,s}$ the complete bipartite graph with partite sets $X$ and $Y$, where $|X|=r$ and $|Y|=s$. The complete bipartite graph $K_{1,n-1}$ is also called the \emph{star graph}. 

 Let $G_1$ and $G_2$ be any two graphs. We define the
 \emph{union} of $G_1$ and $G_2$ to be the graph
 $$G_1 \cup G_2 = (V(G_1)\cup V(G_2), E(G_1)\cup E(G_2)),$$
 and we define the \emph{join} of $G_1$ and $G_2$ to be the graph
 $$
 G_1\vee G_2=(V(G_1)\cup V(G_2), E(G_1)\cup E(G_2)\cup\{uv \mid u\in V(G_1)\text{ and } v\in V(G_2) \}).
 $$

 \begin{prop}
 \label{prop:complement-of-a-union-equals-join-of-complements}
 For any disjoint graphs $G_1$ and $G_2$, we have $\overline{G_1 \cup G_2} = \overline{G_1} \vee \overline{G_2}$.
 \end{prop}

 \begin{proof}
 $V(G_1)$ and $V(G_2)$ induce $\overline{G_1}$ and $\overline{G_2}$ in $\overline{G_1\cup G_2}$. Since $G_1$ and $G_2$ are disjoint, every vertex $u\in V(G_1)$ is joined to every vertex $v\in V(\overline{G_2})$ in $\overline{G_1\cup G_2}$.
 \end{proof}


The following recursive formula will be needed. 

\begin{lem}[\cite{andriantiana2024number}]
\label{lem:recursive-formula}
For any vertex $v$ of a graph $G$
\begin{align*}
\sigma_1(G)
=\sigma_1(G-v)+\sigma_1(G-N_G[v])+\sum_{u\in N_G(v)}\sigma_0(G-(N_G[u]\cup N_G[v])).
\end{align*}
\end{lem}

Recall the following result  established in \cite{andriantiana2024number}.



\begin{thm}[{\rm \cite{andriantiana2024number}}]
\label{thm:if-G-is-a-graph-of-order-n-6-then-sigma-1-at-most-max}
    If $G$ is a graph of order $n \ge 1$, then
    \begin{align*}
        \sigma_1(G) \le  \frac{27}{64}\cdot 2^n,
    \end{align*}
    with equality if and only if $G \in \{ 3K_2 \cup \overline{K_{n-6}}, 4K_2 \cup \overline{K_{n-8}} \}$.
\end{thm}

\begin{rem}
\label{Rem:1}
Theorem \ref{thm:if-G-is-a-graph-of-order-n-6-then-sigma-1-at-most-max} is stated for $n\geq 6$ in \cite{andriantiana2024number}.
From the Tables 1 and 2 in the appendix of the same paper, we have
\begin{align*}
&\max_{|V(G)|=1} \sigma_1(G)=0\leq \frac{27}{64}2^1,
\max_{|V(G)|=2} \sigma_1(G)=1\leq \frac{27}{64}2^2,
\max_{|V(G)|=3} \sigma_1(G)=3\leq \frac{27}{64}2^3,\\
&\max_{|V(G)|=4} \sigma_1(G)=6\leq \frac{27}{64}2^4,
\max_{|V(G)|=5} \sigma_1(G)=12\leq \frac{27}{64}2^5.
\end{align*}
\end{rem}
We now calculate the explicit formulas for the graphs that will be proven to be extremal graphs in the next section. Before we present the explicit formulas, we begin with the following important definition.

\begin{defn}[\cite{andriantiana2024number}]
A graph $G$ is a good graph if for any edge $uv\in E(G)$, we have $N_G[u]\cup N_G[v]=V(G)$.
\end{defn}
If $uv$ is an edge in a good graph $G$, then it can only be contained in exactly one $1$-nearly independent vertex subset which is $\{u,v\}$. Hence, we have the following proposition.
\begin{prop}\cite{andriantiana2024number}
\label{Prop:GD}
For any good graph $G$, we have $\sigma_1(G)=|E(G)|$.
\end{prop}

It is easy to check that $K_n$, $K_{1,n-1}$ and $\overline{3K_3\cup \overline{K_{n-6}}}$ are good graphs.

\begin{prop}
    \label{prop:K-n-and-edgless-equal-min-general-graph}
    For any integer $n\ge 1$, we have
    \begin{align*}
        \sigma_1(K_n) + \sigma_1(\overline{K_n}) = \frac{n(n-1)}{2}.
    \end{align*}
\end{prop}

\begin{proof}
    Since $K_n$ is a good graph, we have
    \begin{align*}
        \sigma_1(K_n) = |E(K_n)| = \binom{n}{2} = \frac{n(n-1)}{2}.
    \end{align*}
    By the definition of $\sigma_1$, we have $\sigma_1(\overline{K_n})=0$ and hence the result follows immediately.
\end{proof}

\begin{prop}
    \label{prop:Star-equal-min-tree-and-connected-graph}
    For any integer $n\ge 2$, we have
    \begin{align*}
        \sigma_1(K_{1,n-1}) + \sigma_1(\overline{K_{1,n-1}}) = (n-1)^2.
    \end{align*}
\end{prop}

\begin{proof}
    Recall that $K_{1,n-1}$ is a good graph, and so
    \begin{align*}
        \sigma_1(K_{1,n-1}) = |E(K_{1,n-1})| = n-1.
    \end{align*}
    Note that $\overline{K_{1,n-1}} \cong K_1 \cup K_{n-1}$, so
    \begin{align*}
        \sigma_1(\overline{K_{1,n-1}}) = \sigma_1(K_1 \cup K_{n-1}) = 2 \cdot \binom{n-1}{2} = (n-1)(n-2).
    \end{align*}
    Thus, we have
    \begin{align*}
        \sigma_1(K_{1,n-1}) + \sigma_1(\overline{K_{1,n-1}}) = n-1 + (n-1)(n-2) = (n-1)^2,
    \end{align*}
    completing the proof.
\end{proof}

\begin{prop}
\label{prop:3K2-equal-max}
    For any integer $n\geq 6$, we have 
        \begin{align*}
            \sigma_1(3K_2 \cup \overline{K_{n-6}}) + \sigma_1(\overline{3K_2 \cup \overline{K_{n-6}}}) = \frac{27}{64}\cdot 2^n + \frac{1}{2}(n+2)(n-3).
        \end{align*}
\end{prop}

\begin{proof}
    By Theorem \ref{thm:if-G-is-a-graph-of-order-n-6-then-sigma-1-at-most-max}, we have $\sigma_1(G) = \sigma_1(3K_2 \cup \overline{K_{n-6}}) = \frac{27}{64}\cdot 2^n$. 



    $\overline{3K_2 \cup \overline{K_{n-6}}}$ is a good graph, and so by Proposition \ref{Prop:GD},
    we have
    \begin{align*}
        \sigma_1(\overline{3K_2\cup \overline{K_{n-6}}}) = \left|E(\overline{3K_2\cup \overline{K_{n-6}}})\right| &= \frac{1}{2} \displaystyle \sum\limits_{v \in V(\overline{3K_2\cup \overline{K_{n-6}}})} \deg (v)\\
        &=\frac{1}{2}\left[ 6(n-2) + (n-6)(n-1)\right]\\
        &=\frac{1}{2}(n^2 - n -6)\\
        &=\frac{1}{2}(n+2)(n-3).
    \end{align*}
    Thus, 
    \begin{align*}
        \sigma_1(3K_2\cup \overline{K_{n-6}}) + \sigma_1(\overline{3K_2\cup \overline{K_{n-6}}}) = \frac{27}{64}\cdot 2^n + \frac{1}{2}(n+2)(n-3),
    \end{align*}
    completing the proof.
\end{proof}


\section{Main result}
\label{sec:main-result}

In this section we present 
a lower and an upper bound on $\sigma_1(G) + \sigma_1(\overline{G})$ if $G$ is a general graph with $n$ vertices and a lower bound if $G$ is a tree with $n$ vertices. The extremal graphs that achieve the equality on each of these bounds are characterised.

\subsection{Minimal general graphs}
\label{subsec:minimal-graphs}

In this subsection, we prove that, among all general graphs of order $n$, the smallest value of $\sigma_1(G) + \sigma_1(\overline{G})$ is only reached by the elements of $\{K_n, \overline{K_n}\}$.

\begin{thm}
\label{thm:if-G-has-order-n-then-NG-at-least-n-choose-2}
    If $G$ is a graph of order $n$, then
    \begin{align}
    \label{ineq:ND-at-least-n-choose-2}
        \sigma_1(G) + \sigma_1(\overline{G}) \ge \frac{n(n-1)}{2},
    \end{align}
    with equality if and only if $G \in \{ K_n , \overline{K_n}\}$.
\end{thm}

\begin{proof}
For any graph $G$ of order $n$, we have
\begin{align*}
	\sigma_1(G)+\sigma_1(\overline{G})\geq |E(G)| + |E(\overline{G})|= |E(K_n)| = \binom{n}{2} =\frac{n(n-1)}{2}.
\end{align*}

To show that equality is attained only if $G\in\{K_n,\overline{K_n}\}$, it is sufficient to show that equality is reached only if $G$ is either a $0$-regular graph or an $(n-1)$-regular graph. Suppose, to the contrary, that $\sigma_1(G)+\sigma_1(\overline{G})=\frac{n(n-1)}{2}$ and $G$ contains a vertex $u$ such that $1\le \deg_G(u)\le n-2$. We show that $G$ or $\overline{G}$ contains a $1$-nearly independent vertex subset of cardinality at least $3$. Since $\deg_G(u) \le n-2$, we have $V(G) \setminus N_G(u) \ne \emptyset$. Let $v \in V(G)\setminus N_G(u)$.

Suppose that there exists a vertex $w \in N_G(u)$ such that $vw \notin E(G)$. Then, the set $\{u,w,v\}$ is a $1$-nearly independent vertex subset of $G$, implying that  $	\sigma_1(G)+\sigma_1(\overline{G})>\frac{n(n-1)}{2}$, a contradiction.

Hence, we may assume that every vertex $w \in N_G(u)$ is adjacent to $v$ in $G$. 
Then $uv\in E(\overline{G})$, $uw \notin E(\overline{G})$ for every $w\in N_G(u)$, and $vw \notin E(\overline{G})$ for every $w\in N_G(u)$. Thus, 
the set $\{u,v,w\}$ is a $1$-nearly independent vertex subset of $\overline{G}$ for every $w\in N_G(u)$, implying that  $	\sigma_1(G)+\sigma_1(\overline{G})>\frac{n(n-1)}{2}$, a contradiction.

We deduce, therefore, that if $G$ is a graph of order $n$ satisfying $\sigma_1(G) + \sigma_1(\overline{G}) = \frac{n(n-1)}{2}$,
then every vertex of $G$ has degree either $0$ or $n-1$ in $G$. If there is a vertex of degree $0$ in $G$, then there is no vertex of degree $n-1$, and thus $G \cong \overline{K_n}$. On the other hand, if there is a vertex of degree $n-1$ in $G$, then there is no vertex of degree $0$, and thus $G\cong K_n$.
\end{proof}

\subsection{Minimal trees}
\label{subsec:Minimal-trees}

We show in this subsection, that among all trees $T$ of order $n$, the star $K_{1,n-1}$ is the only one that reaches the minimum value of  $\sigma_1(T) + \sigma_1(\overline{T})$. 



\begin{thm}
\label{thm:lower-bound-on-trees-is-n-1-squared}
    If $T$ is a tree of order $n$, then
    \begin{align}
        \label{ineq:ND-atleast-n-1-squared}
        \sigma_1(T) + \sigma_1(\overline{T}) \ge (n-1)^2,
    \end{align}
    with equality if and only if $T \cong K_{1, n-1}$.
\end{thm}

\begin{proof}
By Proposition \ref{prop:Star-equal-min-tree-and-connected-graph}, we know that 
$$
\sigma_1(K_{1,n-1})+\sigma_1(\overline{K_{1,n-1}})=(n-1)^2.
$$
We are left to show that for any $n$-vertex tree $T\neq K_{1,n-1}$, we have
$$
\sigma_1(T)+\sigma_1(\overline{T})>(n-1)^2.
$$
Note that 
$$|E(\overline{T})| = \binom{n}{2}-(n-1)=\frac{n(n-1)}{2}-(n-1)=\binom{n-1}{2}.$$

Thus, $T$ and $\overline{T}$ have at least $$n-1 + \binom{n-1}{2}=\binom{n}{2}=(n-1)^2-\binom{n-1}{2}$$ $1$-nearly independent vertex subsets of cardinality $2$. Next, we show that $T$ and $\overline{T}$ together have at least $\binom{n-1}{2}$ $1$-nearly independent vertex subsets with three or more vertices.

Let $uv$ be an edge of $\overline{T}$. If $N_{\overline{T}}[u]\cup N_{\overline{T}}[v]\neq V(\overline{T})$, then $V(\overline{T})\setminus (N_{\overline{T}}[u]\cup N_{\overline{T}}[v]) \ne \emptyset$ and for any vertex $w\in V(\overline{T})\setminus (N_{\overline{T}}[u]\cup N_{\overline{T}}[v])$, the set $\{u,v,w\}$ is a $1$-nearly independent vertex subset of $\overline{T}$.

Now, suppose that $N_{\overline{T}}[u]\cup N_{\overline{T}}[v]= V(\overline{T})$. Note that we cannot have $N_{\overline{T}}(u)= N_{\overline{T}}(v)$, as this would imply that $u$ is isolated in $T$ and making $T$ disconnected, contradicting the fact that it is a tree. So we must have $N_{\overline{T}}(u)\neq N_{\overline{T}}(v)$. Without loss of generality (swapping the labels $u$ and $v$ if needed), we can assume that there is a vertex $y\in N_{\overline{T}}(v)\setminus N_{\overline{T}}(u)$.  But then the set $\{u,v,y\}$ is a $1$-nearly independent subset of $T$, where the edge is $uy$.

Thus, to each of the $\binom{n-1}{2}$ edges in $\overline{T}$, we can associate at least one $1$-nearly independent subset either in $T$ or in $\overline{T}$ with three or more vertices. This means that
\begin{align*}
    \sigma_1(T)+\sigma_1(\overline{T})\geq |E(T)|+|E(\overline{T})|+\binom{n-1}{2} = \binom{n}{2} + \binom{n-1}{2} =(n-1)^2.
\end{align*}
If $T\neq K_{1,n-1}$, then $\sigma_1(T)> |E(T)|=n-1$, and thus $\sigma_1(T)+\sigma_1(\overline{T}) > (n-1)^2$. To see this, for the equality $\sigma_1(T)= |E(T)|$ to be reached, every edge $uv$ of $T$ can only be contained in one $1$-nearly independent vertex subset, which is $\{u,v\}$. So, if $u$ is a leaf in $T$ attached to $v$, then $N_T[u]\cup N_T[v]=N_T[v]=V(T)$. But this would mean that $T=K_{1,n-1}$. 
\end{proof}

\subsection{Maximal general graphs}
\label{subs:maximal-general-graphs}

In this section, we show that the graph $3K_2\cup \overline{K_{n-6}}$ that is proven  in \cite{andriantiana2024number} to have the largest $\sigma_1(G)$ among all $n$-vertex graphs $G$ also has the largest $\sigma_1(G)+\sigma_1(\overline{G})$. 

\begin{thm}
   \label{thm:max-general-graphs}
   If $G$ is a graph of order $n\geq 6$, then
   \begin{align*}
        \sigma_1(G) + \sigma_1(\overline{G}) \le f(n)=\frac{27}{64}\cdot 2^n + \frac{1}{2}(n+2)(n-3),
   \end{align*}
   with equality if and only if $G \in \{ 3K_2\cup \overline{K_{n-6}}, K_{n-6}\vee G_{6,4}\}$, where $G_{6,4}$ is a connected $4$-regular graph of order $6$, i.e., $G_{6,4} \cong \overline{3K_2}$.
\end{thm}

\begin{proof}
    By Proposition \ref{prop:3K2-equal-max}, we know that if $G \in \{ 3K_2\cup \overline{K_{n-6}}, K_{n-6}\vee G_{6,4}\}$, then 
    \begin{align*}
        f(n)=\sigma_1(G) + \sigma_1(\overline{G}) = \frac{27}{64}\cdot 2^n + \frac{1}{2}(n+2)(n-3).
   \end{align*}
   The statement is verified for $6 \le n \le 9$ by exhaustive computation
(see Appendix~\ref{App},  where a SageMath code is provided along with its output).  
Assume the result holds for all graphs of order $6 \le n' < n$, where $n \ge 10$,
and let $G$ be a graph of order $n$.

    In the rest of this proof, $v$ is a vertex of $G$ with $\deg_G(v)=k$. So, $\deg_{\overline{G}}(v)=n-1-k$. By swapping $G$ and $\overline{G}$ if needed, we can assume that $k\leq n-1-k$  and thus $k\leq (n-1)/2,$
that is, $n\geq 2k+1$.

First, we consider the case of $k=0$.

\textbf{Case I: } Suppose that $k=0$, that is, $G$ or $\overline{G}$, say $G$, contains an isolated vertex. 

Then, by Lemma \ref{lem:recursive-formula}, we have
    \begin{align*}
        \sigma_1(G) =\sigma_1(G-v)+\sigma_1(G-N_G[v])+\sum_{u\in N_G(v)}\sigma_0(G-(N_G[u]\cup N_G[v])) = 2 \sigma_1(G-v).
    \end{align*}
    Since the vertex $v$ of degree $0$ in $G$ has degree $n-1$ in $\overline{G}$, again, by Lemma \ref{lem:recursive-formula}, we have
    \begin{align*}
        \sigma_1(\overline{G})&=\sigma_1(\overline{G}-v)+\sigma_1(\overline{G}-N_{\overline{G}}[v])+\sum_{u\in N_{\overline{G}}(v)}\sigma_0(\overline{G}-N_{\overline{G}}[u]\cup N_{\overline{G}}[v])\\
        &=\sigma_1(\overline{G}-v) + 0 + (n-1)(1)\\
        &=\sigma_1(\overline{G}-v) + n-1.
    \end{align*}
    Thus,
    \begin{align*}
        \sigma_1(G) + \sigma_1(\overline{G}) &= 2\sigma_1(G-v) + \sigma_1(\overline{G}-v) + n-1\\
        &=\sigma_1(G-v)  + \sigma_1(\overline{G}-v) + \sigma_1(G-v) + n-1.
    \end{align*}
    By the induction assumption, and by Theorem \ref{thm:if-G-is-a-graph-of-order-n-6-then-sigma-1-at-most-max}, we have
    \begin{align*}
        \sigma_1(G) + \sigma_1(\overline{G}) &\le \frac{27}{64}\cdot 2^{n-1} + \frac{1}{2}(n+1)(n-4) + \frac{27}{64}\cdot 2^{n-1} + (n-1)\\
        &=\frac{27}{64}\cdot 2^n + \frac{1}{2}(n+2)(n-3) + \frac{1}{2}(n+1)(n-4) - \frac{1}{2}(n+2)(n-3) + (n-1)\\
        &=\frac{27}{64}\cdot 2^n + \frac{1}{2}(n+2)(n-3) -(n-1) + (n-1)\\
        &=\frac{27}{64}\cdot 2^n + \frac{1}{2}(n+2)(n-3),     
    \end{align*}
    as desired. Hence, $\sigma_1(G)+\sigma_1(\overline{G})\le f(n)$. Equality occurs only if $G-v\in  \{ 3K_2\cup \overline{K_{n-1-6}}, K_{n-1-6}\vee G_{6,4}\}$. In this case, $G\in  \{ 3K_2\cup \overline{K_{n-6}}, K_{n-6}\vee G_{6,4}\}$. 

In each of the remaining cases, we determine an upper bound $f(n-1)+ C_k(n)$ for $\sigma_1(G)+\sigma_1(\overline{G})$. Then, we show that $C_k(n)\leq f(n)-f(n-1)=\frac{27}{64}\cdot2^{n-1}+ n-1$. For this, we show that $g_n(k)=C_k(n)-(f(n)-f(n-1))\leq 0$. This is done by showing that the derivative of $g_k(n)<0$ and the first value of $g_k(n)$ is non-positive. All of this implies the desired inequality that 
$$
\sigma_1(G)+\sigma_1(\overline{G})\leq f(n).
$$
Note the frequent use of the fact that $\ln(2)>\frac{1}{2}.$

    {\bf Case II:} Suppose that $k\geq 4$.

Note that if a vertex $x \in N_G(v)$ has no neighbour in $V(G)\setminus N_G(v)$ in $G$, then it is adjacent to every element of $N_{\overline{G}}(v)$ in $\overline{G}$. Hence, we may assume that either each element of $N_G(v)$ in $G$ has a neighbour in $V(G)\setminus N_G(v)$ or each element of $N_{\overline{G}}(v)$ in $\overline{G}$ has a neighbour in $V(\overline{G})\setminus N_{\overline{G}}(v)$.

    We can write
    $$
    \deg_G(v)=k\leq n-1-k=\deg_{\overline{G}}  (v)= k+i,
    $$
    for some non-negative integer $i$. For this to be possible, $n\geq 2(4) + 1 =  9$. Thus, if each $y\in N_{\overline{G}}(v)$ has a neighbour in $V(\overline{G})\setminus N_{\overline{G}}(v)$, we have
    \begin{align*}    
\sigma_1(G)+\sigma_1(\overline{G})
    &=\sigma_1(G-v)+\sigma_1(\overline{G}-v)+\sigma_1(G-N_G[v])+\sigma_1(\overline{G}-N_{\overline{G}}[v])\\
    &+\sum_{x\in N_G(v)}\sigma_0(G-N_G[v]-N_G[x]) +\sum_{y\in N_{\overline{G}}(v)}\sigma_0(\overline{G}-N_{\overline{G}}[v]-N_{\overline{G}}[y])\\
    &\leq f(n-1) +\frac{27}{64}\cdot 2^{k+i}+\frac{27}{64}\cdot 2^{k}+ k\cdot 2^{k+i}+(k+i)\cdot 2^{k-1},
    \end{align*}
      and if each $x\in N_G(v)$ has a neighbour in $V(G)\setminus N_G(v)$, we have
     \begin{align*}    
    \sigma_1(G)+\sigma_1(\overline{G})
    &+\sum_{x\in N_G(v)}\sigma_0(G-N_G[v]-N_G[x]) +\sum_{y\in N_{\overline{G}}(v)}\sigma_0(\overline{G}-N_{\overline{G}}[v]-N_{\overline{G}}[y])\\
    &\leq f(n-1) +\frac{27}{64}\cdot 2^{k+i}+\frac{27}{64}\cdot 2^{k}+ k\cdot 2^{k+i-1}+(k+i)\cdot 2^{k}.
    \end{align*}
Since
\begin{align*}
k\cdot 2^{k+i}+(k+i)\cdot 2^{k-1}-\Big( k\cdot 2^{k+i-1}+(k+i)\cdot 2^{k}\Big)
&=2^{k-1}(k\cdot 2^{i+1}+k+i-k\cdot 2^i -2(k+i)\\
&=2^{k-1}(k\cdot 2^i-(k+i))\geq 0,
\end{align*}
for all integers $k\geq 1$ and $i\geq 0$, we have
\begin{align*}    
\sigma_1(G)+\sigma_1(\overline{G})    &\leq f(n-1) +\frac{27}{64}\cdot 2^{k+i}+\frac{27}{64}\cdot 2^{k}+ k\cdot 2^{k+i}+(k+i)\cdot 2^{k-1},
    \end{align*}
    for both cases.
Note that this is under the constraint that $k+k+i=n-1.$

Now, we prove that 
$$
\frac{27}{64}\cdot 2^{k+i}+\frac{27}{64}\cdot 2^{k}+ k\cdot 2^{k+i}+(k+i)\cdot 2^{k-1}< f(n)-f(n-1)=\frac{27}{64}\cdot 2^{n-1}+n-1.
$$
For this, we show that 
$$
g(k,i)=\frac{27}{64}\cdot2^{k+i}+\frac{27}{64}\cdot2^{k}+ k\cdot2^{k+i}+(k+i)\cdot2^{k-1}-\frac{27}{64}\cdot2^{2k+i}-2k-i<0.
$$
First, we study the case where $i=0$. With $n\geq 10$ and $i=0$ we must have $k\geq 4$ to satisfy $n=2k+i+1$.
\begin{align*}
g(k,0)&=\frac{27}{64}\cdot2^{k}+\frac{27}{64}\cdot2^{k}+ k\cdot2^{k}+k\cdot2^{k-1}-\frac{27}{64}\cdot2^{2k}-2k\\
&=\frac{27}{32}\cdot2^{k}+ k\cdot 3\cdot 2^{k-1}-\frac{27}{64}\cdot2^{2k}-2k\\
&=27\cdot 2^{k-5}+k\cdot 3 \cdot 2^{k-1}-27\cdot 2^{2k-6}-2k\\
&=2^{k+1}\left(27\cdot 2^{-6}-27\cdot 2^{k-7}+\frac{3}{4}k\right)-2k\\
&=2^{k+1}\left(\frac{27}{64}(1-2^{k-1})+\frac{3}{4}k\right)-2k.
\end{align*}
The function $h(k)=\frac{27}{64}(1-2^{k-1})+\frac{3}{4}k$ is decreasing for $k\geq 5$ as 
$$
\frac{d}{dk}h(k)=-\frac{27}{128}\cdot2^k \ln(2)   + \frac{3}{4}<0.
$$
Since $h(5)=-165/64<-2$, we have $h(k)\leq  -165/64$ for all $k\geq 5$.
Hence,  $g(k,0)$ is decreasing for $k\geq 5$. Since $g(4,0)=-13/2$ and $g(5,0)=-175$ we deduce that 
$
g(k,0)<0
$
foll $k\geq 4$, which covers all cases of $i=0$ and $n\geq 10.$



Now, we prove that $g(k,i)$ is a decreasing function of $i$.
\begin{align*}
\frac{\partial}{\partial i}g(k,i) &= 
\frac{27}{64}\cdot 2^{k+i}\ln(2) + k \cdot 2^{k+i}\ln(2) + 2^{k-1} - \frac{27}{64}\cdot 2^{2k+i}\ln(2) -1\\
&= 2^{k+i}\ln(2) \left( \frac{27}{64} +k -\frac{27}{64}\cdot 2^k \right) + 2^{k-1} -1<0
\end{align*}
for $k\geq 4$, since $h(k) < -2$ for $k\geq 5$ and $h(4)=3/64$ which leads to $$h(4) - \frac{27}{128}\cdot 2^4 + \frac{1}{4}(4) = -\frac{149}{64} < 0.$$
Thus, $g(k,i)$ is a decreasing function of $k$ for all integers $k\ge 4$ and $i\ge 1$. Since $g(k,1), g(k,0)<0$ for all $k\geq 4$, we now conclude that $g(k,i)<0$ for all $k\geq 4$. Hence,
\begin{align*}
    \sigma_1(G)+\sigma_1(\overline{G}) &\leq f(n-1) +\frac{27}{64}\cdot2^{k+i}+\frac{27}{64}\cdot2^{k}+ k\cdot2^{k+i}+(k+i)\cdot2^{k}\\
    &< f(n-1)+f(n)-f(n-1)=f(n),
    \end{align*}
which completes the proof for this case. 

From now on, we only consider $G$ such that the vertex degrees in $G$ and $\overline{G}$ can only be $k$ or $n-1-k$ for some $0\leq k \le 3$.


{\bf Case III:} Suppose that $k=3$ and $v$ is a vertex of $G$ such that $\deg_G(v) =k=3$. 

Suppose that every element of $N_G(v)$ has at least one neighbour in $V(G)\setminus N_G[v]$. Then
\begin{align*}
\sigma_1(G)+\sigma_1(\overline{G})&=\sigma_1((G-v)+\sigma_1(\overline{G}-v)+\sigma_1(G-N_G[v])+\sigma_1(\overline{G}-N_{\overline{G}}[v])\\
    &+\sum_{x\in N_G(v)}\sigma_0(G-N_G[v]-N_G[x]) +\sum_{y\in N_{\overline{G}}(v)}\sigma_0(\overline{G}-N_{\overline{G}}[v]-N_{\overline{G}}[y])\\
    &\leq f(n-1) +\frac{27}{64}\cdot 2^{n-4}+3+ 3\cdot2^{n-5}+(n-4)\cdot 2^{3}.
    \end{align*}
Thus,
\begin{align*}
g_3(n) &=3+\frac{27}{64}\cdot 2^{n-4}+ 3\cdot2^{n-5}+(n-4)\cdot 2^{3}-\frac{27}{64}\cdot2^{n-1}-(n-1)\\
&=-93\cdot 2^{n - 10} +9n - 30.
\end{align*}
Note that
$$
\frac{d}{dn}g_3(n)=-93\cdot2^{n - 10}\ln(2) + 9 <-93\cdot2^{n - 11} + 9<0
$$
for $n\geq 10$, and $g_3(10)=-33<0$.

Suppose that there is a vertex $u$ in $N_G(v)=\{u,z,w\}$ which does not have a neighbour in $V(G)\setminus N_G(v)$. Then, $u\in N_{\overline{G}}(y)$ for any $y\in N_{\overline{G}}(v)$ and thus
$$
\sum_{y\in N_{\overline{G}}(v)}\sigma_0(\overline{G}-N_{\overline{G}}[v]-N_{\overline{G}}[y])\leq (n-4)\cdot 2^{3-1}.
$$
Since $|V(\overline{G}-N_{\overline{G}}[v])|\leq 3$, we have $\sigma_1(\overline{G}-N_{\overline{G}}[v])\leq 3$. 

If $G-N_G[v]$ has no edges, then in order to satisfy the condition that $n\geq 10$, we must have $\deg_G(z)>3$ for some $z\in N_G(v)$. Thus, $\deg_G(z)\geq n-4$ and so $z$ has at least $n-4-3=n-7$ neighbours in $G-N_G[v]$. Therefore,
$$
\sum_{x\in N_G(v)}\sigma_0(G-N_G[v]-N_G[x]) \leq 2\cdot 2^{n-4} + 2^{n-4-(n-7)}.
$$
Thus, we have
\begin{align*}
\sigma_1(G)+\sigma_1(\overline{G})
    &=\sigma_1((G-v)+\sigma_1(\overline{G}-v)+\sigma_1(G-N_G[v])+\sigma_1(\overline{G}-N_{\overline{G}}[v])\\
    &+\sum_{x\in N_G(v)}\sigma_0(G-N_G[v]-N_G[x]) +\sum_{y\in N_{\overline{G}}(v)}\sigma_0(\overline{G}-N_{\overline{G}}[v]-N_{\overline{G}}[y])\\
    &\leq f(n-1) +\frac{27}{64}\cdot 2^{n-4}+3+ 2 \cdot2^{n-4}+2^{3}+(n-4)\cdot2^{2}.
    \end{align*}
We then have
\begin{align*}
g_3(n)
    &= \frac{27}{64}\cdot 2^{n-4}+3+ 2 \cdot2^{n-4}+2^{3}+(n-4)\cdot2^{2}-\frac{27}{64}\cdot2^{n-1}-(n-1)\\
&=-61\cdot 2^{n - 10} + 3n - 4,
\end{align*}
with
$$
\frac{d}{dn}g_3(n)=-61\cdot2^{n - 10}\ln(2) + 3 <-61\cdot2^{n - 11} + 3<0
$$
for $n\geq 10$, and $g_3(10)=-17<0$.

If $G-N_G[v]$ contains at least one edge, we have
\begin{align*}
\sigma_1(G)+\sigma_1(\overline{G})
    &=\sigma_1((G-v)+\sigma_1(\overline{G}-v)+\sigma_1(G-N_G[v])+\sigma_1(\overline{G}-N_{\overline{G}}[v])\\
    &+\sum_{x\in N_G(v)}\sigma_0(G-N_G[v]-N_G[x]) +\sum_{y\in N_{\overline{G}}(v)}\sigma_0(\overline{G}-N_{\overline{G}}[v]-N_{\overline{G}}[y])\\
    &\leq f(n-1) +\frac{27}{64}\cdot 2^{n-4}+3+ 3\cdot 3\cdot2^{n-4-2}+(n-4)\cdot2^{2}.
    \end{align*}
    For this subcase, $g_3(n)$ is given by
\begin{align*} 
g_3(n)&=\frac{27}{64}\cdot2^{n-4}+ 3+9\cdot2^{n-6}+(n-4)\cdot2^{2}-\frac{27}{64}\cdot 2^{n-1}-(n-1)\\
&=-45\cdot 2^{n - 10} +5n - 14,
\end{align*}
with
$$
\frac{d}{dn}g_3(n)=-45\cdot 2^{n - 10}\ln(2) + 5<-45\cdot 2^{n - 11} +5<0
$$
for $n\geq 10$, and $g_3(10)=-9< 0.$ 
     

%

In each of the wo subcases, we have $\sigma_1(G)+\sigma_1(\overline{G})< f(n)$.

From now on, we assume that $G$ and $\overline{G}$ can only have vertex degrees $k$ of $n-1-k$ for $1\le k\le 2$.

\textbf{Case IV: } $G$ contains a vertex $v$ of degree $1$ with neighbour $u$, and all vertices in $G$ and $\overline{G}$ have degree in $\{1,2,n-2,n-3\}$.
 Then, by Lemma \ref{lem:recursive-formula}, we have
    \begin{align*}
        \sigma_1(G) &=\sigma_1(G-v)+\sigma_1(G-N_G[v])+\sum_{x\in N_G(v)}\sigma_0(G-(N_G[x]\cup N_G[v]))\\
        &= \sigma_1(G-v) + \sigma_1(G-v-u)+\sigma_0(G-v-N_G[u])
    \end{align*}
and
    \begin{align*}
    \sigma_1(\overline{G})&=\sigma_1(\overline{G}-v)+\sigma_1(\overline{G}-N_{\overline{G}}[v])+\sum_{x\in N_{\overline{G}}(v)}\sigma_0(\overline{G}-N_{\overline{G}}[x]\cup N_{\overline{G}}[v])\\
        &=\sigma_1(\overline{G}-v) + 0 + (n-2)(2)\\
        &=\sigma_1(\overline{G}-v) + 2n-4.
    \end{align*}
Suppose that $u$ has degree in $\{n-2,n-3\}$. Then
\begin{align*}
        \sigma_1(G) &= \sigma_1(G-v) + \sigma_1(G-v-u)+\sigma_0(G-v-N_G[u])\\
        &\leq \sigma_1(G-v) + \sigma_1(G-v-u)+4.
    \end{align*}
    
    Thus,
    \begin{align*}
        \sigma_1(G) + \sigma_1(\overline{G}) &= \sigma_1(G-v)+ \sigma_1(G-v-u)+4 + \sigma_1(\overline{G}-v) + 2n-4.
    \end{align*}
    By the induction assumption, and by Theorem \ref{thm:if-G-is-a-graph-of-order-n-6-then-sigma-1-at-most-max}, we have
    \begin{align*}
        \sigma_1(G) + \sigma_1(\overline{G}) &\le \frac{27}{64}\cdot 2^{n-1} + \frac{1}{2}(n+1)(n-4) + \frac{27}{64}\cdot 2^{n-2} + 2n\\
        &=\frac{27}{64}\cdot 2^n + \frac{1}{2}(n+2)(n-3)-\frac{27}{256}2^n + n + 1, \\
        &=\frac{27}{64}\cdot 2^n + \frac{1}{2}(n+2)(n-3)-\frac1{10}2^n + n + 1\\
        &<\frac{27}{64}\cdot 2^n + \frac{1}{2}(n+2)(n-3)
    \end{align*}
    for $n\geq 10$ as desired.

    For the rest of this case, we assume that any vertex $v$ of degree $1$ in $G$ is only adjacent to a vertex $u$ with degree $\{1,2\}$.
    
    Suppose that $\deg_G(u)=1$. Then, $G$ is of the form $K_2\cup G'$ for some $G'$ with order $n-2$. The degrees in $G$ has to be in $\{1,2,n-2,n-3\}$. Note that $G'$ cannot have a vertex of degree $n-2$, as there would be not enough vertices. 

    Suppose that $G'$ has a vertex $w$ of degree $n-3$. Then $N_{G'}(w)=V(G')$. If $G'$ is a star, we can choose a leaf in $G'$ instead of $v$ and get the same case as in the above subcase. So, we can assume that $G'$ is not a star. At least two neighbours $x$ and $y$ of $w$ are adjacent. Let $S'_{n-2}$ be the tree obtained from a star by adding one more edge, then we have
    $$
    \sigma_{0}(G')\leq \sigma_0(S'_{n-2})= 1+3\cdot 2^{n-2-3}.
    $$
    Thus
    \begin{align*}
        \sigma_1(G) &= 3\sigma_1(G-u-v)+\sigma_0(G-u-v)
        \leq 3\sigma_1(G-u-v) +1+3\cdot 2^{n-2-3}
    \end{align*}
 \begin{align*}
        \sigma_1(\overline{G}) &= \sigma_1(\overline{G-u-v})+2(n-2).
    \end{align*}
    
    Thus,
    \begin{align*}
        \sigma_1(G) + \sigma_1(\overline{G}) &= \sigma_1(G-v-u)+\sigma_1(\overline{G-v-u})+ 2\sigma_1(G-v-u)+1+3\cdot 2^{n-5} + 2n-4.
    \end{align*}
    By the induction assumption, and by Theorem \ref{thm:if-G-is-a-graph-of-order-n-6-then-sigma-1-at-most-max}, we have
\begin{align*}
        \sigma_1(G) + \sigma_1(\overline{G}) &\le \frac{27}{64}\cdot 2^{n-2} + \frac{1}{2}n(n-5) + \frac{27}{64}\cdot 2^{n-1} +1+3\cdot 2^{n-5}+2n-4\\
        &=\frac{27}{64}\cdot 2^{n}+\frac12(n+2)(n-3)- \frac{3}{256}\cdot 2^n\\
        &<  \frac{27}{64}\cdot 2^{n}+\frac12(n+2)(n-3).
    \end{align*}
    

    As a last subcase in this case, consider the case when $\deg(u)=2$, say $N_G(u)=\{v,w\}$. Then,
    \begin{align*}
    \sigma_1(G)
    &= \sigma_1(G-v)+\sigma_1(G-v-u)+\sigma_0(G-v-u-w)\\
    \end{align*}
    with 
    \begin{align*}
    \sigma_1(\overline{G})
    &=\sigma_1(\overline{G}-v)+2(n-2).
    \end{align*}
    
    If $G-v-u-w$ does not have an edge, then $\deg_G(w)=n-2$ (recall the restriction that $G$ can only have degrees in $\{1,2,n-2,n-3\}$),
    $$
    \sigma_1(G-v-u)=n-3
    $$
    and
    \begin{align*}
    \sigma_1(G) + \sigma_1(\overline{G})
    &\le \frac{27}{64}\cdot2^{n-1}+ \frac{1}{2}(n+1)(n-4)+n-3+2^{n-3}+2(n-2)\\
    &=\frac{27}{64}2^{n}+ \frac{1}{2}(n+2)(n-3) -\frac{11}{128}\cdot2^n + 2n - 6\\
    &< \frac{27}{64}\cdot 2^{n}+ \frac12(n+2)(n-3)
    \end{align*}
    as $-\frac{11}{128}\cdot2^n + 2n - 6\leq -74$ for $n\geq 10.$

    Now suppose that 
$G-v-u-w$ has an an edge, then 
    $$
    \sigma_0(G-v-u-w)\leq 3\cdot 2^{n-5}
    $$
    and
    \begin{align*}
    \sigma_1(G) + \sigma_1(\overline{G})
    &\le \frac{27}{64}\cdot 2^{n-1}+ \frac12(n+1)(n-4)+ \frac{27}{64}\cdot2^{n-2} +3\cdot 2^{n-5}+2(n-2)\\
    &=\frac{27}{64}\cdot 2^{n}+ \frac12(n+2)(n-3) -\frac{3}{256}\cdot 2^n + n - 3\\
    &< \frac{27}{64}\cdot 2^{n}+ \frac{1}{2}(n+2)(n-3),
    \end{align*}
     as $-\frac{3}{256}\cdot2^n + n - 3\leq -5$ for $n\geq 10.$
    
{\bf Case V:} As a final case, we consider the case when all vertices in both $G$ and $\overline{G}$ can only have degree in $\{2,n-3\}$.

Suppose that $G$ has a vertex $v$ of degree $n-3$. Then
     \begin{align*}
        \sigma_1(G) &=\sigma_1(G-v)+\sigma_1(G-N_G[v])+\sum_{u\in N_G(v)}\sigma_0(G-(N_G[u]\cup N_G[v])) \\
        &\leq \sigma_1(G-v)+1+(n-3)(4).
    \end{align*}

    Let $N_{\overline{G}}(v)=\{u,w\}$. Suppose that one of $u$ and $w$ has degree $n-3$ in $\overline{G}$. Then
 \begin{align*}
        &\sigma_1(\overline{G})=\sigma_1(\overline{G}-v)+\sigma_1(\overline{G}-N_{\overline{G}}[v])+\sum_{z\in N_{\overline{G}}(v)}\sigma_0(\overline{G}-N_{\overline{G}}[z]\cup N_{\overline{G}}[v])\\
        &\leq\sigma_1(\overline{G}-v) + \sigma_1(\overline{G}-v-u-w)  + \sigma_0(\overline{G}-v-u-N_{\overline{G}}[w])  
        +\sigma_0(\overline{G}-v-w-N_{\overline{G}}[u])\\
        &\leq\sigma_1(\overline{G}-v) + \sigma_1(\overline{G}-v-u-w)  + 2^{n-3}+4.
    \end{align*}
Thus,
    \begin{align*}
        \sigma_1(G) + \sigma_1(\overline{G}) 
        &\leq \sigma_1(G-v)  + \sigma_1(\overline{G}-v) +\sigma_1(G-v-u-w)+1+4n-12+2^{n-3}+4\\
        &\leq \frac{27}{64}\cdot2^{n-1}+\frac12(n-1+2)(n-1-3)+\frac{27}{64}\cdot 2^{n-3} +1+4n-12+2^{n-3}+4\\
        &= \frac{27}{64}\cdot 2^{n}+\frac12(n+2)(n-3)-\frac{17}{512}\cdot 2^n + 3n - 6\\
        &<\frac{27}{64}\cdot 2^{n}+\frac12(n+2)(n-3)
    \end{align*}
    as $-\frac{17}{512}\cdot2^n + 3n - 6\leq -10$ for $n\geq 10.$

For the rest of this case, we assume that both $u$ and $w$ have degree $2$ in $\overline{G}$.

    Suppose that $uw\in E(\overline{G})$, then $v,u,w$ form a $C_3$ connected component of $\overline{G})$, that is $\overline{G}=C_3\cup H$ for some graph $H$ of order $n-3\geq 7$. $H$ does not have enough vertices to have a vertex of degree $n-3$. Thus, all vertices in $H$ has degree $2$: $H$ is a cycle and hence contains a $P_4$ subgraph. Thus,
    $$
    \sigma_0(\overline{G}-v-u-w)\leq 8\cdot 2^{n-7}
    $$
    and
    \begin{align*}
        \sigma_1(\overline{G})&=\sigma_1(\overline{G}-v)+\sigma_1(\overline{G}-N_{\overline{G}}[v])+\sum_{z\in N_{\overline{G}}(v)}\sigma_0(\overline{G}-N_{\overline{G}}[z]\cup N_{\overline{G}}[v])\\
        &\leq\sigma_1(\overline{G}-v) + \sigma_1(\overline{G}-v-u-w)  + \sigma_0(\overline{G}-v-u-w)  
        +\sigma_0(\overline{G}-v-w-u)\\
        &\leq\sigma_1(\overline{G}-v) + \sigma_1(\overline{G}-v-u-w)  + 8\cdot 2^{n-6}.
    \end{align*}
    Thus,
    \begin{align*}
        \sigma_1(G) + \sigma_1(\overline{G}) 
        &\leq \sigma_1(G-v)  + \sigma_1(\overline{G}-v) +\sigma_1(G-v-u-w)+1+4n-12+8\cdot 2^{n-6}\\
        &\leq \frac{27}{64}\cdot 2^{n-1}+\frac12(n-1+2)(n-1-3)+\frac{27}{64}\cdot 2^{n-3} +1+4n-12+8\cdot 2^{n-6}\\
        &= \frac{27}{64}\cdot2^{n}+\frac12(n+2)(n-3)-\frac{17}{512}\cdot2^n + 3n - 10\\
        &<\frac{27}{64}\cdot2^{n}+\frac12(n+2)(n-3)
    \end{align*}
    as $-\frac{17}{512}\cdot2^n + 3n - 10\leq -14$ for $n\geq 10.$

     Suppose that $uw \notin E(\overline{G})$, and let $u'$ and $w'$ be the neighbours of $u$ and $w$ other than $v$, respectively. Then
      \begin{align*}
        &\sigma_1(\overline{G})=\sigma_1(\overline{G}-v)+\sigma_1(\overline{G}-N_{\overline{G}}[v])+\sum_{z\in N_{\overline{G}}(v)}\sigma_0(\overline{G}-N_{\overline{G}}[z]\cup N_{\overline{G}}[v])\\
        &\leq\sigma_1(\overline{G}-v) + \sigma_1(\overline{G}-v-u-w)  + \sigma_0(\overline{G}-v-u-w-w')+\sigma_0(\overline{G}-v-w-u-u').
    \end{align*}
    Thus,
    \begin{align*}
        \sigma_1(G) + \sigma_1(\overline{G}) 
        &\leq \sigma_1(G-v)  + \sigma_1(\overline{G}-v) +\sigma_1(G-v-u-w)+1+4n-12+2\cdot 2^{n-4}\\
        &\leq \frac{27}{64}\cdot 2^{n-1}+\frac12(n-1+2)(n-1-3)+\frac{27}{64}\cdot2^{n-3} +1+4n-12+2\cdot 2^{n-4}\\
        &= \frac{27}{64}\cdot2^{n}+\frac12(n+2)(n-3)-\frac{9}{512}\cdot2^n + 2n - 7\\
        &<\frac{27}{64}\cdot2^{n}+\frac12(n+2)(n-3)
    \end{align*}
    as $-\frac{9}{512}\cdot2^n + 2n - 7\leq -5$ for $n\geq 10.$

   This completes the proof of Theorem \ref{thm:max-general-graphs}.
\end{proof}

\newpage

\bibliographystyle{abbrv} 
\bibliography{references}

\medskip 
\appendix
\section{A SAGEMATH code that show all graphs $G$ of order $n$ that have the largest $\sigma_1(G)+\sigma_1(\overline{G})$, for $1\leq n\leq 9$}
\label{App}

\includepdf[
  pages=-,
  pagecommand={\thispagestyle{plain}}
]{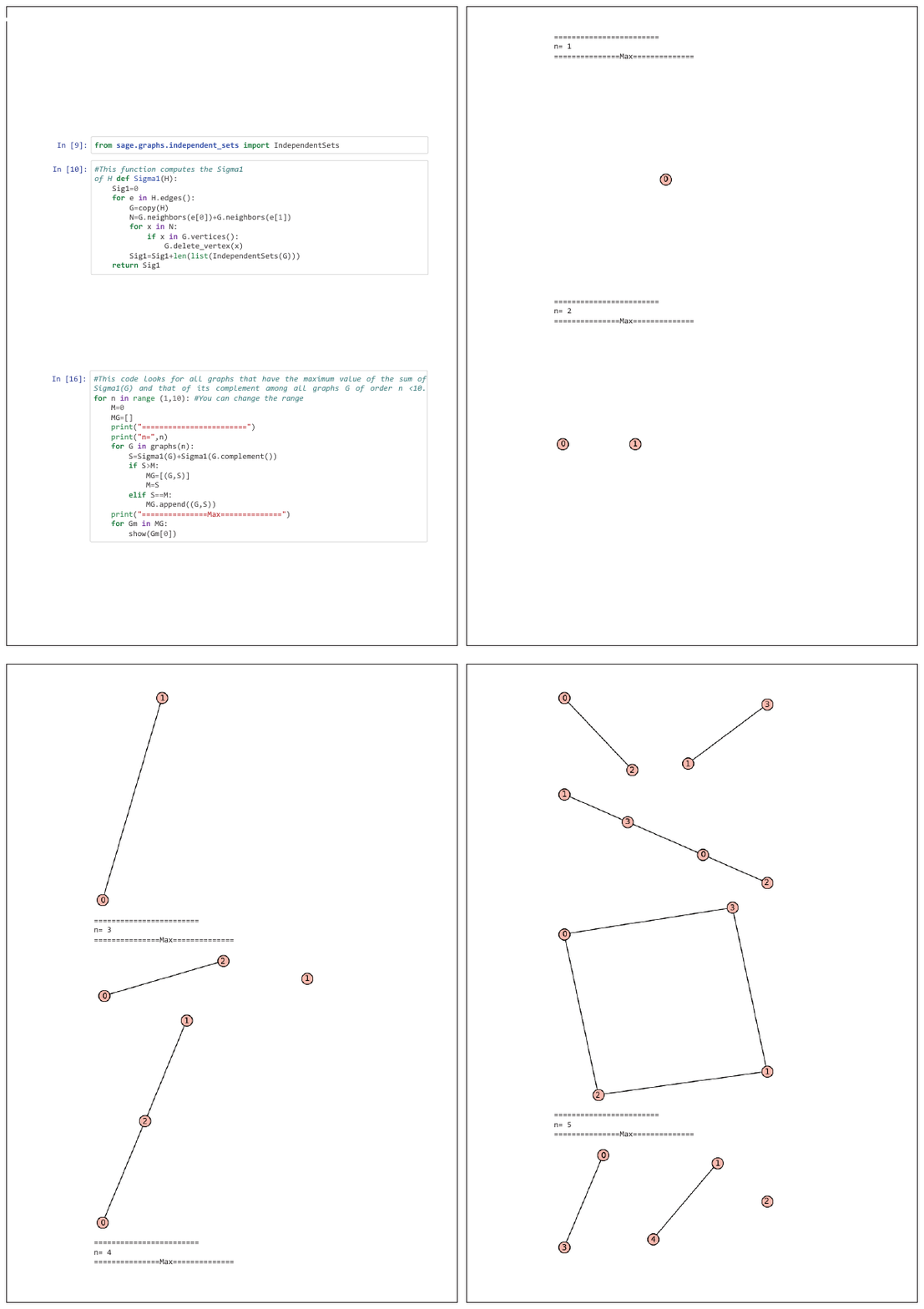}

\end{document}